\documentclass[12pt]{amsart}%
\usepackage{amsmath,amsthm,amsfonts,amssymb}
\usepackage{verbatim}
\usepackage{latexsym}
\usepackage{marginnote}
\usepackage{amscd}
\usepackage{amssymb}
\usepackage{xcolor}
\usepackage{amsmath}
\usepackage{amsfonts}
\usepackage{mathrsfs}
\usepackage{graphicx}%
\providecommand{\U}[1]{\protect\rule{.1in}{.1in}}
\providecommand{\U}[1]{\protect\rule{.1in}{.1in}}
\providecommand{\U}[1]{\protect\rule{.1in}{.1in}}
\newtheorem{theorem}{Theorem}[section]

\newtheorem{definition}[theorem]{Definition}
\newtheorem{corollary}[theorem]{Corollary}

\theoremstyle{definition}

\newcommand\union{{\cup}}
\def\union{\mathop{ \cup}}
\begin{document}
\title{Sobolev spaces on hypergroup Gelfand pairs}
\author{Ky T. Bataka, Murphy  E. Egwe and Yaogan Mensah}
\address{Ky T. Bataka, Department of Mathematics, University of Lom\'e, Togo}
\email{\textcolor[rgb]{0.00,0.00,0.84}{btaka.2005@gmail.com}}
\address{Murphy E. Egwe, Department of Mathematics, University of Ibadan, Nigeria}
\email{\textcolor[rgb]{0.00,0.00,0.84}{egwemurphy@gmail.com, murphy.egwe@ui.edu.ng}}
\address{Yaogan Mensah, Department of Mathematics, University of Lom\'e, Togo}
\email{\textcolor[rgb]{0.00,0.00,0.84}{mensahyaogan2@gmail.com, ymensah@univ-lome.org}}

\begin{abstract}
This paper  introduces Sobolev spaces over Gelfand pairs in the framework of hypergroups. The Sobolev spaces in question are constructed from the Fourier transform on hypergroup Gelfand pairs.  Mainly, the paper focuses on the investigation of  Sobolev embedding results. 
\end{abstract}
\maketitle


\section{Introduction}
Although most Sobolev spaces are defined from the notion of weak derivation, there is a class of Sobolev spaces that has an equivalent definition related to  Fourier transform. For instance, these are the Sobolev spaces  $H^s(\mathbb{R}^n)$ defined as Bessel potential \cite{Behzadan}. More explicitly, $H^s(\mathbb{R}^n)$ is  defined as the set of tempered distributions $u$ on $\mathbb{R}^n$ such that $\|\mathcal{F}^{-1}(\langle \xi\rangle^s\mathcal{F}u)\|_{L^2}<\infty$, where $\langle \xi\rangle=(1+|\xi|^2)^{\frac{1}{2}}$, $s$ is a positive real number and $\mathcal{F}$ denotes the Fourier transform on the space of tempered distributions on $\mathbb{R}^n$. Any abstract generalization of the Fourier transform on groups or other algebraic structures leads to a construction of an analogue of the  Sobolev spaces $H^s(\mathbb{R}^n)$. 

In  \cite{Gorka1, Gorka2}, G\'orka et al. constructed  and studied major properties of the Sobolev spaces $H^s(G)$ where $G$ is a Hausdorff locally compact abelian group using the Fourier transform on locally compact abelian groups. They  proved many continuous embedding theorems and studied the  bosonic string equation. Similar study of Sobolev spaces was done in \cite{Kumar} by M. Kumar and N. S. Kumar  for compact nonnecessarily abelian groups using the Fourier transform on compact group; in other words, they studied some properties of the  Sobolev spaces $H^s(G)$ where $G$ is a compact group. 

In \cite{Krukowski}, Krukowski considered the Sobolev spaces $H^s(K\setminus G/K)$ contructed on  a group Gelfand pair $(G,K)$ by the means of the spherical Fourier transform. This  generalizes the previous works since an abelian group $G$ can be identified with the Gelfand pair $(G, \{e\})$ where $e$ is the neutral element of $G$. Later, in \cite{Mensah},  Mensah constructed and studied the Sobolev $H^{s}_{\delta,\gamma}(G,E)$, where $G$ is a locally compact Hausdorff group and $E$ is a complex Banach space,   by the means of the spherical Fourier transform of type $\delta$ which is a transformation related to a unitary irreducible representation $\delta$ of a compact subgroup $K$ of the locally compact group $G$ forged by Kangni and Tour\'e \cite{Kangni1996, Kangni2001}. When $\delta$ is trivial, then one recovers the classical spherical Fourier transform on  Gelfand pairs  (refer to the book \cite{Wolf}). 

This article proceeds from the same generalization. Our aim is to construct and study the Bessel potential like  Sobolev spaces on hypergroups Gelfand pairs. To achieve  our goals, we relied on the article \cite{Brou} by Brou et al. in which the study of the Fourier transform related to hypergroups Gelfand pairs was considered. As results, we obtain some  analogues of Sobolev embedding theorems. 

The outline of the remaining sections of the paper is as follows. In Section \ref{Preliminaries}, we summarize some major facts about hypergroups Gelfand pairs  and the related Fourier transform.     In Section \ref{Sobolev spaces on the hypergroup Gelfand pair}, we state and prove the main results.
\section{Preliminaries}\label{Preliminaries}
These preliminary notes are based on \cite{Bloom, Brou, Brou2}.
Let $G$ is considered to be  a locally compact Hausdorff space. We denote by 
\begin{itemize}
\item $C(G)$ the set of complex-valued continuous functions on $G$,
\item $M(G)$ the set of Radon measures on $G$,
\item $M_b(G)$ the subset of $M(G)$ consisting of bounded measures, 
\item $M_1(G)$ the subset of  $M_b(G)$ consisting of probability measures,
\item $\mathfrak{C}(G)$ the set of compact subspaces of $G$,
\item $\delta_x$ the point measure at the element $x$ of $G$. 
\end{itemize}
The set $M(G)$ is endowed with the c\^one topology  \cite{Jewett} while  $\mathfrak{C}(G)$ is endowed with  the Michael topology \cite{Michael}.
\begin{definition}[\cite{Brou}]
A locally compact space $G$ is called a {\rm hypergroup} if the following properties hold.
\begin{enumerate}
\item There exists a binary operation $\ast$ (called convolution) on $M_b(G)$ which turns it into an associative algebra such that 
\begin{enumerate}
\item the mapping $(\mu,\nu)\longmapsto \mu\ast\nu$ is continuous from $M_b(G)\times M_b(G)$ into $M_b(G)$,
\item $\forall x,y\in G$, $\delta_x\ast\delta_y$ is a probability measure such that $\mbox{supp}(\delta_x\ast\delta_y)$ is compact.
\item The mapping $(x,y)\longmapsto \mbox{supp}( \delta_x\ast\delta_y)$   is continuous from $G\times G$ into $\mathfrak{C}(G)$.
\end{enumerate}
\item There exists a unique element $e$ in $G$ (called  neutral element) such that 
$$\forall x\in G, \delta_x\ast \delta_e=\delta_e\ast \delta_x=\delta_x.$$
\item There exists an involutive homeomorphism $\diamond: G\longrightarrow G$ such that for all $x,y\in G$, 
$$(\delta_x\ast \delta_y)^\diamond=\delta_{y^\diamond}\ast \delta_{x^\diamond}$$ where $(\delta_x\ast \delta_y)^\diamond(f):=(\delta_x\ast \delta_y )(f^\diamond)$ and $f^\diamond(x)=f(x^\diamond)$ for all $f\in C(G)$. 
\item $\forall x,y,z\in G$, $z\in \mbox{supp}( \delta_x\ast\delta_y)\Longleftrightarrow x\in \mbox{supp}( \delta_z\ast\delta_{y^\diamond})$.
\end{enumerate}
\end{definition}  
As in the group case, one may define the notion of subhypergroup by stability with respect to the involution and convolution.   
\begin{definition}[\cite{Brou2}]
Let $H$ be a closed nonempty subset of the hypergroup $G$. Call $H$ a subhypergroup of $G$ if 
\begin{itemize}
\item $\forall x\in H,\,  x^\diamond\in H$,
\item $\forall x,y\in H,\, (\delta_x\ast \delta_y)\subset H$.
\end{itemize}
\end{definition}
Let $G$ be a hypergroup and let $K$ be a compact subhypergroup of $G$. For $x,y\in G$, $x\ast y$ denotes the support of  $\delta_x\ast \delta_y$. The double coset of $x$ with respect to $K$ is 
$$KxK=\lbrace k_1\ast  x\ast k_2 : k_1,k_2\in K\rbrace=\union\limits_{k_1,k_2\in K}\text{supp}(\delta_{k_1}\ast \delta_x\ast \delta_{k_2}).$$ 
For $f\in C(G)$, set 
$$f(x\ast y):= (\delta_x\ast \delta_y)(f)=\int_Gf(z)d(\delta_x\ast \delta_y)(z).$$
Also, for $\mu,\nu\in M(G)$, define $\mu\ast\nu$ by
$$\mu\ast\nu (f)=\int_G \int_G f(x\ast y)d\mu(x)d\nu (y), \, f\in C(G).$$
A function $f\in C(G)$ is said to be $K$-bi-invariant if 
 $$\forall k_1,k_2\in K, \forall x\in G, f(k_1\ast  x\ast k_2)=f(x).$$
 
Denote by $\mathcal{K}(G)$ the set of continuous functions on $G$ with compact support and by $\mathcal{K}^\natural(G)$ the subset of $\mathcal{K}(G)$ consisting of $K$-bi-invariant functions. 

Now, assume that the hypergroup $G$ is provided with a left Haar measure and $K$ is equipped with a normalized Haar measure. 
For $f\in C(G)$, set 
$$f^\natural (x)=\displaystyle\int_K\int_K f(k_1\ast x\ast k_2)dk_1dk_2.$$  For a measure $\mu\in M(G)$, set $\mu^\natural (f)=\mu(f^\natural),\, f\in \mathcal{K}(G)$. The measure $\mu$ is called $K$-bi-invariant if $\mu^\natural=\mu$. Denote by $M_c^\natural(G)$ the set of complex-valued Radon   measures with compact support that are also $K$-bi-invariant. 
 \begin{definition}
 Let $G$ be a hypergroup and $K$ a compact subhypergroup of $G$. 
  The pair $(G,K)$ is called a  Gelfand pair if the space $(M_c^\natural(G),\ast)$ is commutative.
 \end{definition}
We may refer to this Gelfand pair as a {\it hypergroup Gelfand pair}. 
In the rest of the paper, $(G,K)$ is a hypergroup Gelfand pair. Denote by $\widehat{G^\natural}$ the set of bounded continuous functions  $\phi : G\longrightarrow \mathbb{C}$ such that
 \begin{enumerate}
 \item $\phi$ is $K$-bi-invariant,
 \item $\phi (e)=1$, 
 \item $\forall x,y\in G,\,\displaystyle\int_K \phi(x\ast k\ast y) dk=\phi (x)\phi(y)$,
 \item $\forall x\in G,\,\phi (x^\diamond)=\overline{\phi(x)}$, where $\overline{\phi(x)}$ is the complex conjugate of $\phi(x)$. 
 \end{enumerate}
 The set $\widehat{G^\natural}$ is the dual set of the the hypergroup $G$ \cite{Brou}. When equipped with the topology of uniform convergence on compact sets, $\widehat{G^\natural}$ is a locally compact Hausdorff space. 
 \begin{definition}[\cite{Brou}]
 Let $(G,K)$ be a hypergroup Gelfand pair.
 Let $f\in \mathcal{K}^\natural(G)$. The Fourier transform of $f$ is the map $\widehat{f}: \widehat{G^\natural}\longrightarrow \mathbb{C}$ defined by 
 $$\widehat{f}(\phi)=\displaystyle\int_G \phi (x^\diamond)f(x)dx.$$
 \end{definition}
 By a classical argument, the inverse Fourier transform is given by
 $$f(x)=\displaystyle\int_{\widehat{G^\natural}} \phi (x)\widehat{f}(\phi)d\pi(\phi).$$
 \begin{theorem}[\cite{Brou}]\label{Plancherel}
 Let $(G,K)$ be a hypergroup Gelfand pair. There exists a unique nonnegative measure $\pi$ on $\widehat{G^\natural}$ such that 
 $$\displaystyle\int_G |f(x)|^2dx=\displaystyle\int_{\widehat{G^\natural}}|\widehat{f}(\phi)|^2d\pi(\phi),\, \forall f\in L^1(G)\cap L^2(G).$$
 \end{theorem}

 \section{Sobolev spaces on the hypergroup Gelfand pairs $(G,K)$}\label{Sobolev spaces on the hypergroup Gelfand pair}
 In this section, $(G,K)$ is a hypergroup Gelfand pair. Denote by $L^{2,\natural}(G)$ the space of square integrable (with respect to the Haar measure on $G$)  $K$-bi-invariant complex-valued functions on $G$ and by $L^2(\widehat{G^\natural})$ the space of square integrable complex-valued functions on $G^\natural$ (with respect to the positive measure $\pi$). It follows from the Plancherel  type result (Theorem \ref{Plancherel}) that the Fourier transform can be extended to an isometric isomorphism, from $L^{2,\natural}(G)$ onto $L^2(\widehat{G^\natural})$.
  
   Let $\gamma : \widehat{G^\natural}\longrightarrow \mathbb{R}_+$ be a measurable function  and let $s$ be a positive real. We define the following  Sobolev space. 
 \begin{definition} Let $(G,K)$ be a hypergroup Gelfand pair. The set  
 $$H^{s,\natural}_\gamma (G)=\left\lbrace f\in L^{2,\natural} (G) : \displaystyle\int_{\widehat{G^\natural}}(1+\gamma (\phi)^2)^s|\widehat{f}(\phi)|^2d\pi (\phi)<\infty \right\rbrace$$ provided with the norm  
 $$\|f\|_{H^{s,\natural}_\gamma }=\left(\displaystyle\int_{\widehat{G^\natural}}(1+\gamma (\phi)^2)^s|\widehat{f}(\phi)|^2d\pi (\phi)\right)^{\frac{1}{2}}$$
 will be called a Sobolev space. 
 \end{definition}

 \begin{theorem}
 Let $(G,K)$ be a hypergroup Gelfand pair. Then,  $(H^{s,\natural}_\gamma (G), \|\cdot\|_{H^{s,\natural}_\gamma})$ is a  Banach space. 
 \end{theorem}
 \begin{proof}
 The mapping $f\longmapsto (1+\gamma (\cdot)^2)^{\frac{s}{2}}\widehat{f}(\cdot)$ is an isometric isomorphism from $H^{s,\natural}_\gamma (G)$ onto  $L^2(\widehat{G^\natural})$. Since $L^2(\widehat{G^\natural})$ is a complete space, so is $H^{s,\natural}_\gamma (G)$.
 \end{proof}
 Like the classical case, the space $H^{s,\natural}_\gamma (G)$ comes with a Hilbert space structure as can be seen in the following  corollary.
 \begin{corollary}
 Let $(G,K)$ be a hypergroup Gelfand pair. Then, the space $H^{s,\natural}_\gamma (G)$ is a Hilbert space under the inner product
 $$\langle f,g\rangle_{H^{s,\natural}_\gamma}=\displaystyle\int_{\widehat{G^\natural}}(1+\gamma (\phi)^2)^s\widehat{f}(\phi)\overline{\widehat{g}(\phi)}d\pi (\phi).$$
 \end{corollary}
 In the sequel, the symbol $X \hookrightarrow Y$ will mean that $X$  is continuously embedded in $Y$.
 \begin{theorem} Let $(G,K)$ be a  hypergroup Gelfand pair. Then,
 $H^{s,\natural}_\gamma (G)\hookrightarrow L^{2,\natural} (G)$ with $\|f\|_{L^{2,\natural}}\leqslant \|f\|_{H^{s,\natural}_\gamma}$.
 \end{theorem}
 \begin{proof}Let $f\in H^{s,\natural}_\gamma (G)$. We have
 \begin{align*}
 \|f\|_{L^{2,\natural}}^2&=\displaystyle\int_G |f(x)|^2dx\\
 &=\displaystyle\int_{\widehat{G^\natural}} |\widehat{f}(\phi)|^2d\pi (\phi)\\
 &\leqslant \displaystyle\int_{\widehat{G^\natural}}(1+\gamma (\phi)^2)^s|\widehat{f}(\phi)|^2d\pi (\phi)\\
 &=\|f\|_{H^{s,\natural}_\gamma}^2.
 \end{align*}
 \end{proof}
 
 \begin{theorem}
Let $(G,K)$ be a  hypergroup Gelfand pair. If $s>\sigma>0$,  then
$$H^{s,\natural}_\gamma (G)\hookrightarrow  H^{\sigma,\natural}_\gamma (G)$$ with $\|f\|_{H^{\sigma,\natural}_\gamma}\leqslant \|f\|_{H^{s,\natural}_\gamma}, \forall f\in H^{s,\natural}_\gamma (G)$.
 \end{theorem}
 \begin{proof}
 Let $s>\sigma>0$. Since $1+\gamma (\phi)^2\geqslant 1$, then $(1+\gamma (\phi)^2)^s\geqslant (1+\gamma (\phi)^2)^\sigma$. Therefore, $\|f\|_{H^{s,\natural}_\gamma}\geqslant\|f\|_{H^{\sigma,\natural}_\gamma}$.
 \end{proof}
 The next theorem uses the notion of uniform boundedness. We recall this notion in the following definition.
\begin{definition}[\cite{Rudin}]
 Let $X$ be a set. A family  $\mathcal{F}$ of complex-valued functions on $X$  is said to be uniformly bounded if there exists a real number $M$ such that for all $x\in X$ and for all $f\in \mathcal{F}$, $|f(x)|\leqslant M$.
\end{definition} 
 
 \begin{theorem}
 Let $(G,K)$ be a  hypergroup Gelfand pair.  Assume that $\widehat{G^\natural}$ is unifromly bounded  and that $\displaystyle\frac{1}{1+\gamma (\cdot)^2}\in L^s(\widehat{G^\natural},\pi)$. Then, there exists $C(\gamma,s)>0$ such that  $$\forall f\in H^{s,\natural}_\gamma (G),\,  \|f\|_{\infty}\leqslant C (\gamma,s) \|f\|_{H^{s,\natural}_\gamma}.$$
 \end{theorem}
 \begin{proof}
 \begin{align*}
 |f(x)|&=\left|\displaystyle\int_{\widehat{G^\natural}}\phi (x)\widehat{f}(\phi)d\pi (\phi)\right|\\
 &\leqslant \displaystyle\int_{\widehat{G^\natural}}|\phi (x)||\widehat{f}(\phi)|d\pi (\phi)\\
 &\leqslant \displaystyle\int_{\widehat{G^\natural}}\frac{|\phi (x)|}{(1+\gamma (\phi)^2)^{\frac{s}{2}}}(1+\gamma (\phi)^2)^{\frac{s}{2}}|\widehat{f}(\phi)|d\pi (\phi)\\
 &\leqslant \sup\limits_{\phi\in \widehat{G^\natural}}|\phi (x)|\displaystyle\int_{\widehat{G^\natural}}\frac{1}{(1+\gamma (\phi)^2)^{\frac{s}{2}}}(1+\gamma (\phi)^2)^{\frac{s}{2}}|\widehat{f}(\phi)|d\pi (\phi)\\
 &\leqslant \sup\limits_{\phi\in \widehat{G^\natural}}|\phi (x)|\left( \displaystyle\int_{\widehat{G^\natural}}\frac{d\pi (\phi)}{(1+\gamma (\phi)^2)^s}\right)^{\frac{1}{2}}\left(\displaystyle\int_{\widehat{G^\natural}}(1+\gamma (\phi)^2)^s|\widehat{f}(\phi)|^2d\pi (\phi)\right)^{\frac{1}{2}}\\
 & \text{(H\"older's inequality.)}\\
 &\leqslant  \sup\limits_{\phi\in \widehat{G^\natural}}|\phi (x)|\left\|\frac{1}{1+\gamma (\cdot)}\right\|_{L^s}^{\frac{s}{2}}\|f\|_{H^{s,\natural}_\gamma}.
 \end{align*}
 Since $\widehat{G^\natural}$ is uniformly bounded, then 
 $\sup\limits_{x\in G}\sup\limits_{\phi\in \widehat{G^\natural}}|\phi (x)|<\infty$. 
 Therefore, $$\sup\limits_{x\in G}|f(x)|\leqslant \sup\limits_{x\in G}\sup\limits_{\phi\in \widehat{G^\natural}}|\phi (x)| \left\|\frac{1}{1+\gamma (\cdot)}\right\|_{L^s}^{\frac{s}{2}}\|f\|_{H^{s,\natural}_\gamma}.$$
 If we set $C(\gamma,s)=\sup\limits_{x\in G}\sup\limits_{\phi\in \widehat{G^\natural}}|\phi (x)|\left\|\frac{1}{1+\gamma (\cdot)}\right\|_{L^s}^{\frac{s}{2}}$, then $\|f\|_{L^\infty}\leqslant C (\gamma,s) \|f\|_{H^{s,\natural}_\gamma}$, which completes the proof.  
 \end{proof}
 The next theorem uses the notion of equicontinuity. We recal this notion in the following definition. 
 \begin{definition}[\cite{Choquet}]
 Let $X$ be a topological space and $Y$ a metric space (with metric denoted by $d$). Let $\mathcal{F}$ be a family of  functions from $X$ into $Y$ that are continuous.  The family $\mathcal{F}$ is said to be equicontinuous at $a\in X$ if for each $\varepsilon>0$, there exists a neighborhood $V_a$ of $a$ such that for all $x\in X$ and for all $f\in \mathcal{F}$,  
 $$x\in V_a\Longrightarrow d(f(a),f(x))<\varepsilon.$$
\end{definition}
 \begin{theorem}Let $(G,K)$ be a  hypergroup Gelfand pair. Assume that  $\widehat{G^\natural}$ is equicontinuous and 
  $(1+\gamma (\cdot)^2)^{-\frac{s}{2}}\in L^2(\widehat{G^\natural})$. If $f\in H^{s,\natural}_\gamma (G)$,  then $f$ is continuous. 
  \end{theorem}

 \begin{proof}
 Let $\varepsilon>0$ and let $a\in G$. Since $\widehat{G^\natural}$ is equicontinuous, there exists a neighbourhood $U$ of $a$ in $G$ such that $\forall\phi \in\widehat{G^\natural}, \forall x\in G$, $x\in U$ implies $|\phi(x)-\phi(a)|<\varepsilon$. Let $f\in H^{s,\natural}_\gamma (G)$. Then, 
 \begin{align*}
 |f(x)-f(a)|&=\left|\displaystyle \int_{\widehat{G^\natural}}\phi (x)\widehat{f}(\phi)d\pi (\phi)-\int_{\widehat{G^\natural}}\phi (a)\widehat{f}(\phi)d\pi (\phi)\right|\\
 &\leqslant \displaystyle \int_{\widehat{G^\natural}}|\phi(x)-\phi(a)||\widehat{f}(\phi)|d\pi(\phi)\\
 &\leqslant \varepsilon\displaystyle \int_{\widehat{G^\natural}}|\widehat{f}(\phi)|d\pi(\phi)\\
 &=\varepsilon\displaystyle \int_{\widehat{G^\natural}}(1+\gamma(\phi)^2)^{\frac{s}{2}}|\widehat{f}(\phi)|(1+\gamma(\phi)^2)^{-\frac{s}{2}}d\pi(\phi)\\
 &\leqslant\varepsilon \left(\int_{\widehat{G^\natural}} (1+\gamma(\phi)^2)^{s}|\widehat{f}(\phi)|^2d\pi(\phi)\right)^{\frac{1}{2}}\left( \int_{\widehat{G^\natural}}(1+\gamma(\phi)^2)^{-s}d\pi(\phi)  \right)^{\frac{1}{2}}\\
 & \leqslant\varepsilon \|f\|_{H^{s,\natural}_\gamma}\|(1+\gamma (\cdot)^2)^{-\frac{s}{2}}\|_{L^2}.
 \end{align*}
 Thus, $f$ is continuous. 
 \end{proof}
 \begin{theorem}
Let $(G,K)$ be a  hypergroup Gelfand pair. Assume that  $\widehat{G^\natural}$ is uniformly bounded  and 
  $(1+\gamma (\cdot)^2)^{-\frac{s}{2}}\in L^2(\widehat{G^\natural})$. If $f\in H^{s,\natural}_\gamma (G)$,  then $f$ is bounded and there exists $C(\gamma,s)>0$ such that   
$$\|f\|_{\infty}\leqslant C(\gamma,s) \|f\|_{H^{s,\natural}_\gamma}.$$
 \end{theorem}
 \begin{proof}
 Let $f$ be in $H^{s,\natural}_\gamma (G)$. 
 \begin{align*}
 |f(x)|&=\left|\displaystyle \int_{\widehat{G^\natural}}\phi(x)\widehat{f}(\phi)d\pi(\phi)\right|\\
 &\leqslant \displaystyle \int_{\widehat{G^\natural}}|\phi(x)||\widehat{f}(\phi)|d\pi(\phi)\\
 &\leqslant\sup\limits_{\phi\in \widehat{G^\natural}}|\phi (x)|\displaystyle \int_{\widehat{G^\natural}}(1+\gamma(\phi)^2)^{\frac{s}{2}}|\widehat{f}(\phi)|(1+\gamma(\phi)^2)^{-\frac{s}{2}}d\pi(\phi)\\
 &\leqslant\sup\limits_{\phi\in \widehat{G^\natural}}|\phi (x)|\left(\displaystyle \int_{\widehat{G^\natural}}(1+\gamma(\phi)^2)^s|\widehat{f}(\phi)|^2d\pi(\phi)\right)^{\frac{1}{2}}\left(\displaystyle \int_{\widehat{G^\natural}} (1+\gamma(\phi)^2)^{-s}d\pi(\phi)  \right)^{\frac{1}{2}}\\
 & \leqslant \sup\limits_{\phi\in \widehat{G^\natural}}|\phi (x)|\|f\|_{H^{s,\natural}_\gamma}\|1+\gamma(\cdot)^2)^{-\frac{s}{2}}\|_{L^2}. 
 \end{align*}
 Since $\widehat{G^\natural}$ is uniformly bounded, there exists a real $M>0$ such that for all $\phi\in \widehat{G^\natural}$ and for all $x\in G$, $|\phi(x)|<M$. Therefore, 
 $$|f(x)|\leqslant M\|f\|_{H^{s,\natural}_\gamma}\|1+\gamma(\cdot)^2)^{-\frac{s}{2}}\|_{L^2}.$$ Thus $f$ is bounded  and $$\|f\|_{\infty}\leqslant C(\gamma,s) \|f\|_{H^{s,\natural}_\gamma}$$ by taking $C(\gamma,s)=M\|1+\gamma(\cdot)^2)^{-\frac{s}{2}}\|_{L^2}.$
 \end{proof}
\section*{Conclusion} 
Sobolev spaces related to the Fourier transform on hypergroup Gelfand pair are constructed and their major properties are studied. More precisely, embedding results are proved.

\end{document}